\documentclass[11pt]{article}
\usepackage{latexsym,amsmath,stackrel,color,amsthm,amssymb,epsfig,graphicx,mathrsfs}
\usepackage{graphicx}
\usepackage{amssymb}
\usepackage[left=1in,top=1in,right=1in,bottom=1in]{geometry}
\usepackage[linktocpage=true]{hyperref}
\usepackage{setspace}
\usepackage{amssymb, amsmath, amsthm, graphicx,mathrsfs}
\usepackage{caption}

\usepackage{comment,caption}
\usepackage{tikz}
\makeatletter
\def\thm@space@setup{%
  \thm@preskip=\parskip \thm@postskip=0pt
}
\makeatother
\def\qed{\hfill\ifhmode\unskip\nobreak\fi\quad\ifmmode\Box\else\hfill$\Box$\fi}
\def\ite#1{\hfill\break${}$\hbox to 50pt {\quad(#1)\hfill}}

\newtheorem{thm}{Theorem}[section]
\newtheorem{cor}[thm]{Corollary}

\newtheorem{definition}{Definition}

\newtheorem{lem}[thm]{Lemma}
\newtheorem{conjecture}{Conjecture}

\newtheorem{prop}{Proposition}[section]

\def\ex{{\rm{ex}}}

\newcommand{\cir}{\circlearrowright}
\newcommand{\phir}{{r \choose \lfloor r/2 \rfloor}}
\def\cF{{\mathcal F}}
\def\cI{{\mathcal I}}
\def\cJ{{\mathcal J}}
\begin{document}

\pagestyle{myheadings}
\markright{{\small{\sc F\"uredi, Jiang, Kostochka, Mubayi, and Verstra\"ete:   Crossing paths, JULY 12
}}}

\title{
\vspace{-0.8in}Extremal problems on ordered and convex geometric hypergraphs}

\author{
\hspace{0.8in} Zolt\'an F\" uredi\thanks{Research supported by grant K116769
from the National Research, Development and Innovation Office NKFIH and
by the Simons Foundation Collaboration grant \#317487.}
\and
Tao Jiang\thanks{Research partially supported by National Science Foundation award DMS-1400249.}
\and
Alexandr Kostochka\thanks{Research  supported in part by NSF grant
 DMS-1600592 and by grants 18-01-00353A and 16-01-00499
  of the Russian Foundation for Basic Research.
} \hspace{0.8in} \smallskip \and
Dhruv Mubayi\thanks{Research partially supported by NSF awards DMS-1300138 and 1763317.} \and Jacques Verstra\"ete\thanks{Research supported by NSF award DMS-1556524.}
}

\maketitle

\vspace{-0.5in}

\begin{abstract}
An ordered hypergraph is a hypergraph whose vertex set is linearly ordered, and a convex geometric hypergraph is
a hypergraph whose vertex set is cyclically ordered. Extremal problems for ordered and convex geometric graphs have a rich history with applications to a variety of problems in combinatorial geometry. In this paper, we consider analogous extremal problems for uniform hypergraphs, and discover a general partitioning phenomenon which allows us to determine the order of magnitude
of the extremal function for various ordered and convex geometric hypergraphs. A special case is the ordered
$n$-vertex $r$-graph $F$ consisting of two disjoint sets $e$ and $f$ whose vertices alternate in the ordering. We show that for all $n \geq 2r + 1$, the maximum number of edges in an ordered $n$-vertex $r$-graph not containing $F$ is exactly
\[ {n \choose r} - {n - r  \choose r}.\]
This could be considered as an ordered version of the Erd\H{o}s-Ko-Rado Theorem, and generalizes earlier results of Capoyleas and Pach and Aronov-Dujmovi\v{c}-Morin-Ooms-da Silveira.
\end{abstract}

\section{Introduction}

An {\em ordered graph} is a graph together with a linear ordering of its vertex set. Extremal problems for ordered graphs have a long history, and were studied extensively in papers by Pach and Tardos~\cite{PT}, Tardos~\cite{T} and Kor\'{a}ndi, Tardos, Tomon and Weidert~\cite{KTTW}. Let $\ex_{\rightarrow}(n,F)$ denote the maximum number of edges in an $n$-vertex ordered graph that does not contain the ordered graph $F$. This extremal problem is phrased in~\cite{KTTW} in terms of pattern-avoiding matrices. Marcus and Tardos~\cite{MT} showed that if the forbidden pattern is
a permutation matrix, then the answer is in fact linear in $n$, and thereby solved the Stanley-Wilf Conjecture,
as well as a number of other well-known open problems. A central open problem in the areas was posed by
Pach and Tardos~\cite{PT}, in the form of the following conjecture:

\begin{conjecture}\label{ptc}
Let $F$ be an ordered acyclic graph with interval chromatic number two. Then $\ex_{\rightarrow}(n,F) = O(n\cdot \mbox{\rm polylog} \,n)$.
\end{conjecture}

In support of Conjecture A, Kor\'{a}ndi, Tardos, Tomon and Weidert~\cite{KTTW} proved for a wide class of forests $F$ that
$\ex_{\rightarrow}(n,F) = n^{1 + o(1)}$. This conjecture is related to a question of Bra{\ss} in the context of
convex geometric graphs.

A {\em convex geometric graph} is a graph together with a cyclic ordering of its vertex set. Given a convex geometric graph $F$, let $\ex_{\circlearrowright}(n,F)$ denote the maximum number of edges in an $n$-vertex convex geometric graph that does not contain $F$. Extremal problems for geometric graphs have a fairly long history, going back to theorems on disjoint line segments~\cite{Hopf-Pannwitz,Sutherland,Kupitz-Perles}, and more recent results on crossing matchings~\cite{Brass-Karolyi-Valtr,Capoyleas-Pach}. Motivated by the famous Erd\H{o}s unit distance problem, the first author~\cite{Furedi} showed that the maximum number of unit
distances between points of a convex $n$-gon is $O(n\log n)$. In the vein of Conjecture \ref{ptc},
Bra{\ss}~\cite{Brass} asked for the determination of all acyclic graphs $F$ such that $\ex_{\circlearrowright}(n,F)$ is linear in $n$, and this problem remains open.

In this paper, we study extremal problems for ordered and convex geometric uniform hypergraphs. An {\em ordered $r$-graph} is an $r$-uniform hypergraph whose vertex set is linearly ordered. A {\em convex geometric $r$-graph} is an $r$-uniform hypergraph whose vertex set is cyclically ordered. We denote by $\ex_{\rightarrow}(n,F)$ the maximum number of edges in an $n$-vertex ordered $r$-graph that does not contain $F$, and let $\ex(n,F)$ denote the usual (unordered) extremal function. Similarly we write $\ex_{\circlearrowright}(n,F)$  in the convex geometric hypergraph setting. As is the case for convex geometric graphs, the extremal problems for convex geometric hypergraphs are frequently motivated by problems in discrete geometry~\cite{Brass-Rote-Swanepoel, Pach-Pinchasi,Brass,Aronov}. Instances of the extremal problem for two disjoint triangles in the convex geometric setting are connected to the well-known triangle-removal problem~\cite{Gowers-Long}. In~\cite{FJKMV} we show that certain types of paths in the convex geometric setting give the current best bounds for the notorious extremal problem for tight paths in uniform hypergraphs. One of the goals of this paper is to show similarities and differences in solutions of an extremal problem in linearly ordered and cyclically ordered settings.

\section{Results}
\subsection{A splitting theorem}

Given subsets $A, B$ of an ordered set, write $A<B$ to mean that $a<b$ for each $a \in A$ and $b \in B$. For $k \ge r \ge 2$, an  ordered $r$-graph  has {\em 
	interval chromatic number} $k$ if its vertex set can be partitioned into $k$ sets $A_1<A_2<\cdots < A_k$ such that every edge has at most one vertex in each $A_i$. Of particular interest to us is the case $k=r$, when the sets $A_i$ give an $r$-partition of the $r$-graph.  

Let $z_{\rightarrow}(n,F)$ denote the maximum number of edges in an $n$-vertex ordered $r$-graph of interval chromatic number $r$ that does not contain the ordered graph $F$.
Pach and Tardos~\cite{PT} showed that any $n$-vertex ordered graph may be written as a union of at most $\lceil \log n\rceil$ edge disjoint subgraphs each of whose components is a graph of interval chromatic number two, and deduced for every ordered graph $F$ that $\ex_{\rightarrow}(n,F) = O(z_{\rightarrow}(n,F)\log n)$. Our first result generalizes their result to hypergraphs. 

\begin{thm}\label{splitting1}
Fix $r \ge c\ge r-1\ge 1$ and an ordered $r$-graph $F$ with $z_{\rightarrow}(n, F)=\Omega(n^{c})$. Then 
 \[ \ex_{\to}(n, F) = \left\{\begin{array}{ll}
O(z_{\to}(n, F) \log n) & \mbox{ if } c =r-1 \\
O(z_{\to}(n, F))& \mbox{ if }c >r-1. 
\end{array}\right.\]
\end{thm}

We will give a short self-contained proof of Theorem~\ref{splitting1}, although it also follows quickly from our next result, which is the main new ingredient in this work. 

\begin{definition}
	An ordered $r$-graph $F$ is a {\em split hypergraph} if there is a partition of $V(F)$ into intervals $X_1<X_2<\dots<X_{r - 1}$ and there exists $i \in [r-1]$ such that  every edge of $F$ has two vertices in $X_i$ and one vertex in every $X_j$ for $j \ne i$.
\end{definition} For instance, every $r$-graph of interval chromatic number $r$ is a split hypergraph. We write $e(H)$ for the number of edges in a hypergraph $H$, $v(H) = \bigl|\bigcup_{e \in H} e\bigr|$ and $d(H) = e(H)/v(H)^{r - 1}$ for the codegree density of $H$. 

\begin{thm}\label{splitting}
	For every $r\geq 3$ there exists $c=c_r>0$ such that every ordered $r$-graph $H$ contains a split subgraph $G$  with $d(G) \geq c\,d(H)$.
\end{thm}

In the next section, we describe an application of Theorems \ref{splitting1} and \ref{splitting} to extremal problems for ordered $r$-graphs, which demonstrates that loss of the factor $\log n$ between $\ex_{\rightarrow}(n,F)$ and $z_{\rightarrow}(n,F)$ is sometimes necessary. This example will also reveal a discrepancy 
  between the extremal functions for an ordered $r$-graph in the ordered setting versus the convex geometric setting.

\subsection{Crossing paths}

A {\em tight $k$-path} is an $r$-graph whose edges have the form
$\{v_i,v_{i + 1},\dots,v_{i + r - 1}\}$ for $0 \leq i < k$. Typically, we list the vertices $v_0v_1\dots v_{k+r-2}$ in a tight $k$-path.
We consider ordered tight paths to which Theorem \ref{splitting} applies, and for which we obtain the exact ordered extremal function in a number of cases. We let $<$ denote the underlying ordering of the vertices of an ordered or convex geometric hypergraph.

\begin{definition} [Crossing paths] \label{defCP} An {\em $r$-uniform crossing $k$-path} $CP_k^r$ is a tight $k$-path $v_0v_1\dots v_{r+k-2}$
with the ordering
\vspace{-0.1in}
\begin{center}
\begin{tabular}{lp{5.8in}}
{\rm (i)} & $v_0 < v_1 < v_2 < \dots < v_{r-1}$,  \\
{\rm (ii)} &  $v_j  <  v_{j+r}  < v_{j + 2r}  <  \cdots <  v_{j+1}$ for $j<r-1$ and\\
{\rm (iii)} & $v_{r-1} < v_{2r-1} < v_{3r-1}<\cdots$.
\end{tabular}
\end{center}
\end{definition}

An example of an ordered $CP_5^2$  (Figure 1) and of a convex geometric  $CP_7^2$ and $CP_5^3$  (Figure 2) are shown below.

\begin{figure}[!ht]
	\begin{center}
		\includegraphics[width=3.5in]{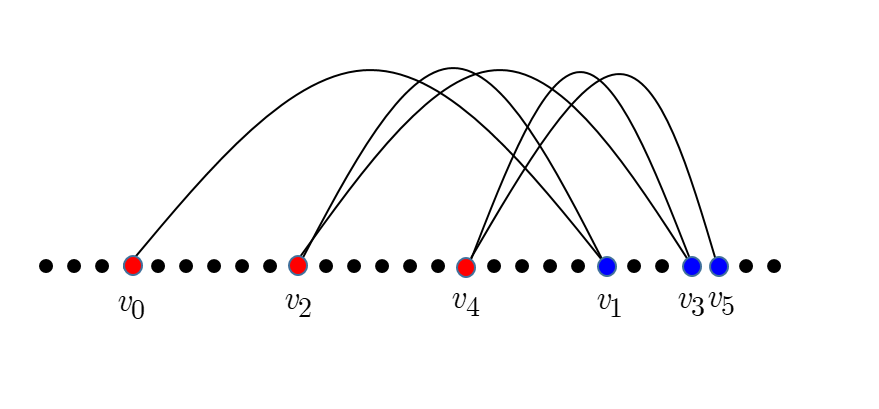}
		\caption{Ordered $CP_5^2$}
		\label{fig:pathcrossers}
	\end{center}
\end{figure}

\begin{figure}[!ht]
\begin{center}
\includegraphics[width=3.5in]{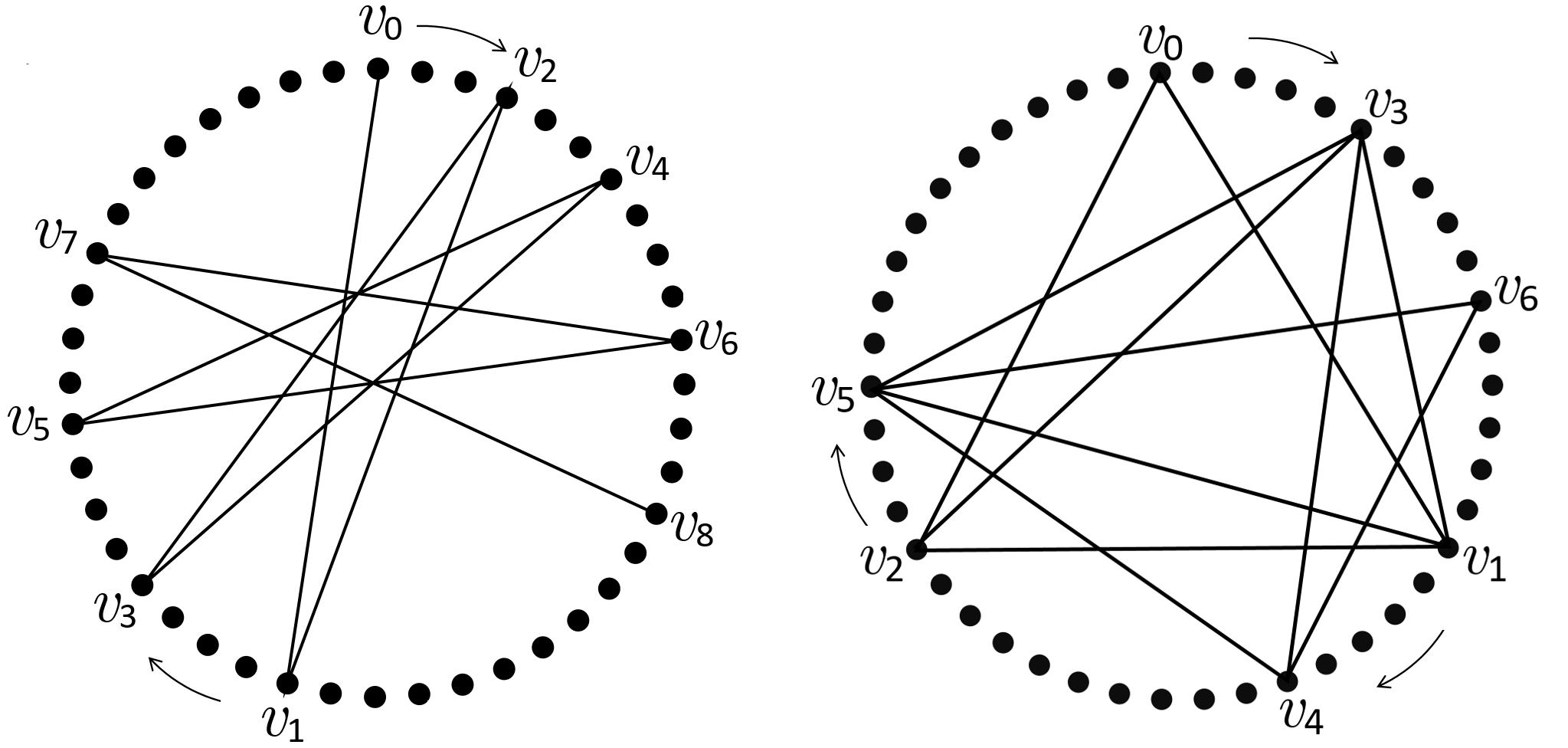}
\caption{Convex Geometric $CP_7^2$ and $CP_5^3$}
\label{fig:crossers}
\end{center}
\end{figure}

Our first result determines the order of magnitude of the extremal function for crossing paths in the ordered setting, and the exact extremal function for short crossing paths. We note that there are very few exact results known for ordered graphs or hypergraphs.

\begin{thm}\label{cpthm}
Let $k \geq 1$, $r \geq 2$ and $n\geq r+k$. Then
$$\ex_{\to}(n, CP^r_k)=  \begin{cases}
{n \choose r} - {n-k+1 \choose r} &\mbox{ for } k \leq r + 1  \\
\Theta(n^{r - 1}\log n) & \mbox{ for }k \geq r + 2.
\end{cases}$$

\end{thm}

Theorem \ref{cpthm} for $k \geq r + 2$ shows that the $\log n$ factor in Theorem \ref{splitting} is necessary, as we shall see for all $k,r \geq 2$ that $z_{\rightarrow}(n,CP_k^r) = O(n^{r-1})$.

In the convex geometric setting, Bra{\ss}, K\'{a}rolyi and Valtr~\cite{Brass-Karolyi-Valtr} proved that $\ex_{\cir}(n,CP_3^2) = 2n - 3$ for $n \geq 3$. 
We generalize this to $CP_k^r$ for $r>2$ and $k>3$  in the following theorem:

\begin{thm}\label{cgthm} 
Let $k \geq 1$, $r \geq 2$ and $n \ge 2r+1$. Then 
  $$ \ex_{\circlearrowright}(n, CP^r_k) = \begin{cases}
  \Theta(n^{r - 1}) & \mbox{ for } k \leq 2r-1 \\
  {n \choose r} - {n - r \choose r} & \mbox{ for }k = r + 1 \\
  \Theta(n^{r - 1}\log n) & \mbox{ for }k \geq 2r.
  \end{cases}$$

  \end{thm}
  
This reveals a discrepancy between the ordered setting and the convex geometric setting: in the convex geometric setting, crossing paths of length up to $2r - 1$ have extremal function of order $n^{r - 1}$, whereas this phenomenon only occurs for crossing paths of length up to $r + 1$ in the ordered setting. In fact, we know  that  
$\ex_{\circlearrowright}(n, CP^r_k)=\ex_{\to}(n, CP^r_k)$ iff $k \in \{1, r+1\}$. The proofs of Theorems \ref{cpthm} and \ref{cgthm} rely substantially on Theorems~\ref{splitting1} and \ref{splitting}.

Theorem \ref{cpthm} has a simple corollary for crossing matchings: a {\em crossing matching} $CM^r$ consists of two disjoint $r$-sets $\{v_0,v_2,v_4,\dots,v_{2r-2}\}$ and $\{v_1,v_3,\dots,v_{2r-1}\}$ such that $v_0 < v_1 < v_2 < \dots < v_{2r-1}$. In this way, Theorem \ref{cpthm} could be viewed as an ordered version of the Erd\H{o}s-Ko-Rado Theorem. Aronov, Dujmovi\v{c}, Morin, Ooms and da Silveira~\cite{Aronov} showed that $\ex_{\circlearrowright}(n,CM^3) = \Theta(n^{2})$ and Capoyleas and Pach~\cite{Capoyleas-Pach} proved an exact result for the convext geometric graph comprising $k$ pairwise crossing line segments. Starting with the simple observation that the ordered crossing path $CP_{r+1}^r$ contains $CM^r$, and also that $\ex_{\rightarrow}(n, CM^r) = \ex_{\circlearrowright}(n,CM^r)$, we
obtain the following corollary to Theorem \ref{cpthm} for $k = r + 1$:

\begin{cor}\label{cmcor}
For $n>r>1$, $ \ex_{\rightarrow}(n, CM^r) = \ex_{\circlearrowright}(n,CM^r) = {n \choose r} - {n-r \choose r}.$
\end{cor}

We shall see that the same convex geometric $r$-graph which does not contain $CP_{r+1}^r$ used to prove Theorem \ref{cpthm} also does not contain $CM^r$, which establishes the equality in the corollary.

\section{Proof of Theorem~\ref{splitting1}}

In this section,  we suppose that the underlying set (the set of vertices) of an ordered hypergraph is $[n]$.
An {\em interval} is a set of consecutive vertices in the ordering.
Given a set of intervals $I_1< I_2< \dots < I_r$
a {\em box} $B(I_1, \dots, ,I_r)$ is a set of (ordered) $r$-sets $\{ x_1, x_2, \dots, x_r\}$ such that $x_i\in I_i$.
We say that a box  $B(I_1, \dots, ,I_r)$ is  {\em covered by} (or contained in) the box $B(J_1, \dots, , J_r)$ if $I_t\subseteq J_t$ for all $t\in [r]$.
A {\em weighted $r$-uniform hypergraph} on a set $X$ is a function $\omega : {X \choose r} \rightarrow [0,\infty)$. 
For a family $\cF$, let
$w(\cF):= \sum_{F\in \cF} w(F)$.  Theorem~\ref{splitting1} follows from the following more general result.

\begin{thm} \label{lognthm2} Let $r\ge c\ge r-1\ge 1$ and let $\omega : {[n] \choose r} \rightarrow [0,\infty)$ be
a weighted $r$-uniform hypergraph.
	Suppose that there is some $A>0$ such that $w(B)\leq A \ell^c$
	for every box $B(I_1, \dots, ,I_r)$
	with $|I_1|=\dots =|I_r|=\ell$. 
	Then	
	 \[w\left(\binom{[n]}{r}\right)  < \left\{\begin{array}{ll}
C A  n^{r-1} \log n	& \mbox{ if } c=r-1 \\
 C  A n^c & \mbox{ if }c>r-1,
	\end{array}\right.\]
	where the  $C$ depends only on $r$ in the first case and only  on $r$ and $c$ in the second case.
\end{thm}

\proof 
Since the statement is monotone, to avoid ceilings and floors, for easier presentation we suppose that $n=r^g$ for some integer $g\geq 1$.
Define a system of intervals $\cI_1, \dots, \cI_g$ and systems of boxes $\cJ_1, \dots, \cJ_{g}$ as follows.
The system $\cI_t$ is obtained by splitting $[n]$ into $r^t$ equal intervals. So $|\cI_t|= r^t$ and each member of it has length $n/r^t$.
For any family of (disjoint) intervals $\cI$, let $B^r(\cI)$ (or just $B(\cI)$) denote the family of boxes of dimension $r$ with intervals from $\cI$.
The family $\cJ_1$ consists of a single box, $\cJ_1:=B(\cI_1)$.
For $t>1$, let $\cJ_t$ be the set of boxes from $B(\cI_t)$ that are {\em not} covered by any member of $B(\cI_{t-1})$.
Since $B(\cI_g)=\binom{[n]}{r}$,  the boxes $\cJ_1\cup \dots \cup \cJ_q$ cover the whole hypergraph.

By definition, $|\cJ_1|=1$. For $t>1$ we can give a (generous) upper bound for the size of $|\cJ_t|$ as follows:
The $r$ intervals from $\cI_t$ defining a member of $\cJ_t$ cannot be spread out into $r$ intervals of $\cI_{t-1}$.
So first, select two subintervals of a member of $\cI_{t-1}$ and then arbitrarily other $(r-2)$ members of $\cI_t$.
One can do this in at most
\[   |\cJ_t|   \leq |\cI_{t-1}|\binom{r}{2} \binom{r^t}{r-2} < r^t\times r^2\times r^{t(r-2)}= r^{t(r-1)+1}
\]
different ways.
The weight of each box from $\cJ_t$ is bounded above by $A(n/r^t)^c$. Hence
\[ \sum w(\cJ_t) \leq \sum_{1\leq t\leq g} (Arn^c) r^{t(r-1-c)}.
\qed
\]

\section{Proof of Theorem~\ref{splitting}}

Throughout this section, $H$ is a convex geometric $n$-vertex $r$-graph, with cyclic ordering $<$ on the vertices. A subgraph $G$ of $H$ is a {\em split subgraph} if there exists a partition of $V(G)$ into cyclic intervals $X_1, X_2, \dots, X_{r - 1}$ such that for some $i \in [r - 1]$, every edge $e$ of $G$ has two vertices in $X_i$ and one vertex in every $X_j : j \neq i$. Let $v(H) = \bigl|\bigcup_{e \in H} e\bigr|$ and $d(H) = e(H)/v(H)^{r - 1}$ denote the codegree density of $H$. Our goal is to prove the following Theorem.

\begin{thm}\label{splitter}
For every $r\geq 3$ there exists $c \ge r^{-5r^2}$ such that every convex geometric $r$-graph $H$ contains a split subgraph $G$  with $d(G) \geq c\,d(H)$.
\end{thm}

We make no attempt to determine the optimal value of the constant $c$ in this theorem; it is not hard to show that $c = e^{-\Omega(r)}$.
It is straightforward to derive Theorem \ref{splitting} from this theorem.

\subsection{Weighted hypergraphs}

The proof of Theorem \ref{splitter} is inductive, and for the induction to work, we appeal to weighted hypergraphs
defined in the previous section.
 The $r$-sets of positive weight form a
hypergraph on $X$ which we denote by $H(\omega)$, and we let $V(\omega)$ be the union of all edges in $H(\omega)$ and
\[ |\omega| = \sum_{e \in H(\omega)} \omega(e).\]
We may think of $V(\omega)$ as the {\em vertex set} of $H(\omega)$, and we let $v(\omega) = |V(\omega)|$. When the range of $\omega$ is $\{0,1\}$, then $|\omega| = |H(\omega)|$ is the number of edges in $H(\omega)$.
Furthermore, for any $r$-graph $H$ on $X$, if $\omega(e) = 1$ if $e \in H$ and $\omega(e) = 0$ otherwise, then $H(\omega) = H$, so any hypergraph can be realized as a weighted hypergraph. The {\em codegree density} of $\omega$ is defined by \begin{equation}\label{wdef}
d(\omega) = \frac{|\omega|}{v(\omega)^{r - 1}}.
\end{equation}
If $G$ is a subgraph of $H(\omega)$, let $\omega_G$ be defined by $\omega_G(e) = \omega(e)$ for $e \in G$ and
$\omega_G(e) = 0$ otherwise. This is the {\em restriction} of $\omega$ to $G$. Note that if $\omega : X \rightarrow \{0,1\}$, then the codegree density of $\omega$
is exactly the codegree density of $H(\omega)$. We obtain Theorem \ref{splitting} for an ordered $r$-graph $H$ by defining $\omega(e) = 1$ for $e \in E(H)$ and $\omega(e) = 0$ otherwise.

\subsection{Bipartite subgraphs}

Let us say a convex geometric hypergraph $H$ is {\em bipartite} if there exists an interval $X$ such that every edge of $H$ has exactly one vertex in $X$. We first prove
a lemma on bipartite subgraphs of convex geometric $r$-graphs, and then use the lemma to commence a proof of a weighted generalization of Theorem \ref{splitter} by induction on $r$.

\begin{lem}\label{b}
Let $r \geq 3$ and let $\omega$ be a weighted convex geometric $r$-graph. Then there exists a bipartite $G \subseteq H(\omega)$ such that
$d(\omega_G) \geq d(\omega)/r^{5r}$.
\end{lem}

\begin{proof}
The proof of the lemma is by induction on $v(\omega)$. If $v(\omega) \leq r^5$, then  we can set $G$ to be an $r$-set of maximum weight.
To see this, note that $\omega_G(e) \geq |\omega|/{v(\omega) \choose r} \geq r! |\omega|/v(\omega)^r$.
Since $v(\omega_G) \leq v(\omega)$,
\[ d(\omega_G) = \frac{|\omega_G|}{v(\omega_G)^{r - 1}} \geq \frac{r! |\omega|}{v(\omega)^{2r-1}} \geq \frac{r!}{v(\omega)^r} \cdot d(\omega).\]
Finally, use the fact that $v(\omega) \leq r^5$ to get $d(\omega_G) \geq d(\omega)/r^{5r}$, as required.

 Suppose $v(\omega) > r^5$. Partition $V(H(\omega))$ into $r$ intervals $X_1,X_2,\dots,X_r$ such that $|X_1| \leq |X_2| \leq \dots \leq |X_r| \leq |X_1| + 1$. Let $H_j$ be the bipartite subgraph of all edges of $H$ with exactly one vertex in $X_j$ and put $\omega_j = \omega_{H_j}$. If $|\omega_j| \geq |\omega|/r^{5r}$ for some $j \in [r]$, then $d(\omega_j) \geq d(\omega)/r^{5r}$, and $G = H_j$ is the required bipartite subgraph. If $|\omega_j| < |\omega|/r^{5r}$ for all $j \in [r]$, let $F = H(\omega) \backslash \bigcup_{j = 1}^r H_j$. Then
\[
|\omega_F| \geq |\omega| - \sum_{j = 1}^r |\omega_j| > \Bigl(1 - \frac{1}{r^{5r-1}}\Bigr) \cdot |\omega| = c_r \cdot |\omega|.
\]For $S \subset [r]$ of size $\lfloor r/2 \rfloor$, let $\omega_S$ be the weighted hypergraph defined by $\omega_S(e) = \omega(e)$ if $|e \cap X_j| \geq 2$ for every $j \in S$, and $\omega_S(e) = 0$ otherwise.
By the pigeonhole principle,
\[|\omega_S| > \frac{|\omega_F|}{\phir} > \frac{c_r}{\phir} \cdot |\omega|
\]
for some $S \subset [r]$ of size $\lfloor r/2\rfloor$. Now every edge in $H(\omega_S)$ is disjoint from every $X_j : j \not \in S$, so
\[
v(\omega_S) \leq v(\omega) - \sum_{j \not \in S} |X_j| \leq v(\omega) - \sum_{j \not \in S} \Big\lfloor\frac{v(\omega)}{r} \Big\rfloor \leq \frac{v(\omega) + r}{2}.
\]
By induction, there is a bipartite $G \subseteq H(\omega_S)$ such that $d(\omega_G) \geq d(\omega_S)/r^{5r}$. Finally,
\begin{eqnarray*}
d(\omega_S) &=& \frac{|\omega_S|}{v(\omega_S)^{r-1}} \; \; \geq \; \;  \frac{c_r (2n)^{r - 1}}{\phir(v(\omega) + r)^{r-1}} \cdot d(\omega).
\end{eqnarray*}
It suffices to show that this is at least $d(\omega)$, and so $d(\omega_G) \geq d(\omega)/r^{5r}$.
To see this, let $n=v(\omega)$, note that $(1 + r/n)^{r - 1} \leq e^{1/r^3} \leq e^{1/27}$ when $n > r^5$, and therefore
\[ \phir(n + r)^{r - 1} \leq e^{1/27}\phir n^{r - 1}.\]
Next note $\phir \leq 2^{r}\sqrt{2}/\sqrt{\pi r} < 2^r \cdot \sqrt{2}/3$, and therefore
\[ \phir(n + r)^{r - 1} \leq \frac{\sqrt{8}}{3} \cdot e^{1/27} \cdot (2n)^{r - 1} < 0.99 (2n)^{r - 1}.\]
Now $c_r \geq c_3 > 1 - 3^{-14} > 0.99$ and so the proposition is proved.
\end{proof}

\subsection{Proof of Theorem \ref{splitter}}

Using Lemma \ref{b} and  induction on $r$, we prove the following generalization of Theorem \ref{splitter}. Set $\psi(r) = r^{5r^2}$.


\begin{thm}\label{weightedversion}
Let $r\geq 3$ and let $\omega$ be a weighted convex geometric $r$-graph. Then there exist a split $r$-graph $G \subseteq H(\omega)$ such that $d(\omega_G) \geq d(\omega)/\psi(r)$.
\end{thm}

\vspace{-0.1in}

\begin{proof}
Proceed by induction on $r$. For $r = 3$, Lemma \ref{b} give the theorem, since in that case a bipartite subgraph is a split subgraph. For $r > 3$,
pass to a bipartite $F \subseteq H(\omega)$ with $d(\omega_F) \geq d(\omega)/r^{5r}$ via Lemma \ref{b}. Let $X$ and $Y$ be the parts of $F$, where every edge of $F$ intersects
$X$ in exactly one vertex. Let
\[
 F_1 = \{e \backslash \{x\} : e \in F, x \in X\}
 \]
and define the new weight function $\tau$ by $\tau(f) = 0$ if $f \not \in F_1$ and for $f \in F_1$,
\[
 \tau(f) = \sum_{{e \in F}\atop{f \subset e}} \omega(e).
 \]
We note that $F_1 = F_1(\tau)$ and $|\tau| = |\omega_F|$. Since $v(\tau) \leq v(\omega_F)$,
\begin{eqnarray*}
d(\tau) = \frac{|\tau|}{v(\tau)^{r - 2}} \geq \frac{v(\omega_F)|\omega_F|}{v(\omega_F)^{r - 1}} = v(\omega_F)d(\omega_F).
\end{eqnarray*}
Note that $F_1$ is $(r - 1)$-uniform, which accounts for the appearance of the extra factor $v(\omega_F)$.
Using $d(\omega_F) \geq d(\omega)/r^{5r}$, we find
\[ d(\tau) \geq v(\omega_F) \cdot \frac{d(\omega)}{r^{5r}}.\]
By induction, there exists an almost $r$-partite subgraph $E \subseteq F_1$ such that
\begin{equation}\label{induction}
 d(\tau_{E}) \geq \frac{d(\tau)}{\psi(r - 1)} \geq v(\omega_F) \cdot \frac{d(\omega)}{r^{5r} \psi(r - 1)}.
 \end{equation}
Let $Z_2,Z_3,\dots,Z_{r-1}$ be the parts of $E$, and let $Z_1 = X$ if $v(\tau_E) \geq |X|$, otherwise let $Z_1$ be a uniformly selected subset of $X$ of size $v(\tau_E)$.
Now we define the subgraph we want: let
\[
G = \{e \cup \{z\} : e \in E, z \in Z_1\}.
\]
We claim that with positive probability, $G$ is the required almost $r$-partite subgraph, with parts $Z_1,Z_2,\dots,Z_{r-1}$. We first prove the following technical proposition:

\begin{prop} Let $m = \min\{v(\tau_E),|X|\}$. Then
\begin{equation}\label{claim}
\frac{m \cdot v(\tau_E)^{r - 2}v(\omega_F)}{|X| \cdot v(\omega_G)^{r-1}} \geq 2^{-(r - 1)}.
\end{equation}
\end{prop}
To see this, if $m = v(\tau_E)$, then $v(\omega_F) \geq |X|$ and $v(\omega_G) = 2m$, so
\[ m \cdot v(\tau_E)^{r - 2}v(\omega_F) \geq m^{r - 1} \cdot |X| \geq |X| \cdot (v(\omega_G)/2)^{r - 1},\]
as required for (\ref{claim}). Otherwise, $m = |X|$ and $v(\omega_F) \geq v(\omega_G)$ and $v(\tau_E) \geq v(\omega_G)/2$,  so
\[ m \cdot v(\tau_E)^{r-2}v(\omega_F) \geq |X| \cdot v(\tau_E)^{r - 2}v(\omega_G) \geq |X| \cdot (v(\omega_G)/2)^{r-1}\]
which proves the proposition. \qed

\medskip

By linearity of expectation,
\begin{equation}\label{expect}
\mathbb E(|\omega_G|) = \frac{m \cdot |\tau_E|}{|X|}.
\end{equation}
Fix an instance of $G$ with $|\omega_G| \geq \mathbb E(|\omega_G|)$. Then by (\ref{induction}) and (\ref{expect}),
\begin{eqnarray*}
d(\omega_G) \; \; = \; \; \frac{|\omega_G|}{v(\omega_G)^{r - 1}} &\geq& \frac{m \cdot |\tau_E|}{|X| \cdot v(\omega_G)^{r - 1}} \\
&=& \frac{m \cdot v(\tau_E)^{r-2}}{|X| \cdot v(\omega_G)^{r - 1}} \cdot d(\tau_E)\\
&\geq&   \frac{m \cdot v(\tau_E)^{r-2}v(\omega_F)}{|X| \cdot v(\omega_G)^{r-2}} \cdot \frac{d(\omega)}{r^{5r}\psi(r - 1)}.
\end{eqnarray*}
Using (\ref{claim}), we obtain
\[ d(\omega_G) \geq \frac{d(\omega)}{2^{r-1} \cdot r^{5r} \psi(r - 1)}.\]
To complete the proof of Theorem \ref{weightedversion}, it suffices to prove that
$2^{r - 1} r^{5r} \psi(r - 1) \leq \psi(r)$. This follows from
\[ 2^{r - 1} r^{5r} \psi(r - 1) \leq 2^{r} r^{5r} r^{5r^2 - 10r + 5} \leq 2^r r^{5-5r} \psi(r)\]
since $r \geq 3$, $r^{5 - 5r} \leq 2^{-r}$.
\end{proof}

 Now Theorem \ref{splitter} follows from Theorem \ref{weightedversion} by setting $\omega(e) = 1$ for all $e \in H$ and $\omega(e) = 0$ otherwise, in which case $d(H) = d(\omega)$.

\section{Proof of Theorem~\ref{cpthm}}

\subsection{Upper bound for $k \leq r + 1$}
We start with the following recurrence:

\begin{prop} \label{firstclaim}
Let $2\leq k\leq r+1$ and $n\geq r+k$. Then
\begin{equation}\label{j12}
\ex_{\rightarrow}(n, CP^r_k)\leq {n-2\choose r-2}+\ex_{\rightarrow}(n-2, CP^{r-1}_{k-1})+\ex_{\rightarrow}(n-1, CP^r_k).
\end{equation}
\end{prop}
\proof Let $G$ be an $n$-vertex ordered $r$-graph  not
containing  $CP^r_k$ with $e(G)=\ex_{\rightarrow}(n, CP^r_k)$. We may assume $V(G)=[n]$ with the natural ordering.
Let $G_1=\{e\in G: \{1,2\}\subset e\}$ and $G_2=\{e\in G: 1\in e, 2\notin e, e-\{1\}\cup \{2\}\in G\}$.
Let $G_3$ be obtained from $G-E(G_1)-E(G_2)$ by
 gluing vertex $1$ with vertex $2$ into
a new vertex $2'$.

Since we have deleted the edges of $G_1$, our $G_3$ is an $r$-graph,
 and since we have deleted the edges of $G_2$,  $G_3$ has no multiple edges.
Thus $e(G)=e(G_1)+e(G_2)+e(G_3)$.

 We view $G_3$ as an ordered $r$-graph with vertex set $\{2',3,\ldots,n\}$. If $G_3$ contains a crossing ordered path
 $P$ with edges $e'_1,e'_2,\ldots,e'_k$, then only $e'_1$ may contain $2'$, and all other edges are edges of $G$.
 Thus either $P$ itself is in $G$ or the path obtained from $P$ by replacing $e'_1$ with $e'_1-\{2'\}+\{1\}$ or with
$e'_1-\{2'\}+\{2\}$ is in $G$, a contradiction. Thus $G_3$ contains no $CP^r_k$ and hence
\begin{equation}\label{j13}
e(G_3)\leq \ex_{\rightarrow}(n-1, CP^r_k).
\end{equation}

By definition, $e(G_1)\leq {n-2\choose r-2}$. We can construct an ordered $(r-1)$-graph $H_2$ with vertex set
$\{3,4,\ldots,n\}$ from $G_2$ by deleting from each edge vertex $1$. If $H_2$ contains a crossing ordered path
 $P'$ with edges $e''_1,e''_2,\ldots,e''_{k-1}$, then the set of edges $\{e_1,\ldots,e_k\}$ where $e_1=e''_1+\{1\}$ and
 $e_i=e''_{i-1}+\{2\}$ for $i=2,\ldots,k$ forms a $CP^r_k$ in $G$, a contradiction. Summarizing, we get
\begin{eqnarray*}
\ex_{\rightarrow}(n, CP^r_k)=e(G) &=& e(G_1)+e(G_2)+e(G_3) \\
&\leq&  {n-2\choose r-2}+\ex_{\rightarrow}(n-2, CP^{r-1}_{k-1})+ \ex_{\rightarrow}(n-1, CP^r_k),
\end{eqnarray*}
as claimed.\qed

We are now ready to prove the upper bound in Theorem \ref{cpthm} for $k \leq r + 1$: We are to show that
$\ex_{\rightarrow}(n, CP^r_k) \leq {n \choose r} - {n - k + 1 \choose r}$.
We use induction on $k+n$. Since $CP^r_1$ is simply an edge, $\ex_{\rightarrow}(n, CP^r_1)=0$ for any $n$ and $r$, and the theorem holds for $k=1$. 

Suppose now the upper bound in the theorem holds for all $(k',n',r')$ with $k'+n'<k+n$ and we want to prove it for $(k,n,r)$. By the previous paragraph,
it is enough to consider the case $k\geq 2$. Then by Proposition \ref{firstclaim} and the induction assumption,
\begin{eqnarray*}\ex_{\rightarrow}(n, CP^r_k) &\leq& {n-2\choose r-2}+
\left[{n-2 \choose r-1} - {n-k \choose r-1}\right]+\left[ {n-1 \choose r} - {n-k \choose r}\right] \\
&=&\left[{n-2\choose r-2}  +{n-2 \choose r-1}+{n-1 \choose r}\right] -
\left[  {n-k \choose r}+ {n-k \choose r-1}\right] \\
&=& {n \choose r} - {n-k+1 \choose r},
\end{eqnarray*}
as required. This proves the upper bound in Theorem \ref{cpthm} for $k \leq r + 1$. \qed

\subsection{Lower bound for $k \leq r + 1$.} 

For the lower bound in Theorem \ref{cpthm} for $k \leq r + 1$, we provide the following construction. For  $1\leq k\leq r$, let $G(n,r,k)$ be the family of $r$-tuples $(a_1,\ldots,a_r)$ of positive integers such that
\begin{center}
\begin{tabular}{lp{5in}}
$(a)$ & $1\leq a_1<a_2<\ldots<a_r\leq n$ and \\
$(b)$ & there is $1\leq i\leq k-1$ such that $a_{i+1}=a_i+1$.
\end{tabular}
\end{center}
Also, let $G(n,r,r+1)=G(n,r,r)\cup \{(a_1,\ldots,a_r): a_1<a_2<\ldots<a_r=n\}$.

\medskip

Suppose $G(n,r,k)$ has a crossing ordered path with edges $e_1,\ldots,e_k$.
Let $e_1=(a_1,\ldots,a_r)$ where $1\leq a_1<a_2<\ldots<a_r\leq n$.
By the definition of a crossing ordered path, for each $2\leq j\leq k$, $e_j$ has the form
\begin{equation}\label{j14}
\mbox{\em
$e_j=(a_{j,1},\ldots,a_{j,r})$ where $a_i< a_{j,i}<a_{i+1}$ for $1\leq i\leq j-1$ and $a_{j,i}=a_{i}$ for $j\leq i\leq r$.}
\end{equation}
By the definition of $G(n,r,k)$, either there is $1\leq i\leq k-1$ such that $a_{i+1}=a_i+1$ or $k=r+1$ and $a_r=n$.
In the first case, we get a contradiction with~(\ref{j14}) for $j=i+1$. In the second case,
we get a contradiction with~(\ref{j14}) for $j=r+1$.

In order to calculate $|G(n,r,k)|$, consider the following procedure $\Pi(n,r,k)$ of generating all $r$-tuples of elements of $[n]$ {\em not} in
$G(n,r,k)$: take an  $r$-tuple $(a_1,\ldots,a_r)$
of positive integers such that
 $1\leq a_1<a_2<\ldots<a_r\leq n-k+1$ and then increase $a_j$ by $j-1$ if $1\leq j\leq k$ and by $k-1$ if
 $k\leq j\leq r$. By definition, the number of outcomes of this procedure is ${n-k+1\choose r}$. Also  $\Pi(n,r,k)$
 never generates a member of $G(n,r,k)$ and generates each other $r$-subset of $[n]$ exactly once.  \qed

\subsection{Upper bound for $k \geq r + 2$}

In this section we apply Theorem~\ref{splitting1} to prove the upper bound in Theorem~\ref{cpthm} for $k \geq r + 2$.
This follows quickly from the following proposition:

\begin{prop} \label{secondclaim}
For $k \geq 1$, $r \geq 2$, $z_{\rightarrow}(n,CP_k^r) = O(n^{r - 1})$.
\end{prop}

\proof We prove a stronger statement by induction on $k$: if $H$ is an ordered $n$-vertex $r$-graph with an interval $r$-coloring
with parts $X_1,X_2,\dots,X_r$ of size $n_1,n_2,\dots,n_r$, and $H$ has no crossing $k$-path, then
$$e(H) \le k \cdot \prod_{i = 1}^r n_i \cdot \sum_{i=1}^r \frac{1}{n_i}.$$
Let $f(k)$ be this upper bound and let $P = \prod_{i = 1}^r n_i$.
The base case $k=1$ is trivial. For the induction step, assume the result holds for paths of length at most $k-1$, and suppose $e(H) > f(k)$. For each $(r-1)$-set $S$ of vertices mark the edge $S \cup \{w\}$ where $w$ is maximum.  Let $H'$ be the $r$-graph of unmarked edges. Since we marked at most $f(k)/k$ edges, $e(H') > f(k-1)$. By the  induction
assumption there exists a $CP^r_{k-1} =v_1 v_2 \ldots v_{k+r-2}  \subset H'$ and we can extend this to a $CP^r_k$ in $H$ using the marked edge obtained from the $(r-1)$-set $\{v_{k}, \ldots, v_{k+r-2}\}$. This proves the proposition. \qed

Proposition \ref{secondclaim} and Theorem~\ref{splitting1} give $\ex_{\rightarrow}(n, CP_k^r) = O(n^{r-1}\log n)$ for all $k \ge 2$ as required.

\subsection{Lower bound for $k \geq r + 2$}\label{orderedconstruction}

We now turn to lower bound in Theorem~\ref{cpthm}.
Let $G(n,r,r+2)$ be the family of $r$-tuples $(a_1,\ldots,a_r)$
of positive integers such that
\begin{center}
\begin{tabular}{lp{5in}}
$(a)$ & $1\leq a_1<a_2<\ldots<a_r\leq n$ and \\
$(b)$ & $a_{2}-a_1=2^p$, where $p\leq \log_2 (n/4)$ is an integer.
\end{tabular}
\end{center}
The number of choices of $a_1\leq n/4$ is $n/4$,  then the number of choices of $a_2$ is $\log_2 (n/4)$, and
the number of choices of the remaining $(r-2)$-tuple $(a_3,\ldots,a_r)$ is at least ${n/2\choose r-2}$.
Thus if $r\geq 3$ and $n>20r$, then
\begin{equation}\label{j15}
|G(n,r,r+2)|\geq \frac{n^{r-1}}{(r-2)!3^{r}}\log_2 n.
\end{equation}

Suppose $G(n,r,r+2)$ contains a $CP_{r+2}^r$ with  vertex set  $\{a_1, \ldots, a_{2r+1}\}$ and edge set
 $\{a_i\ldots a_{i+r-1}: 1 \le i \le r+2\}$. By the definition of ordered path, the vertices are in the following order on $[n]$:
  \begin{equation}\label{j16}
  a_1<a_{r+1}<a_{2r+1}<a_2<a_{r+2}<a_3<a_{r+3}<\ldots<a_r<a_{2r}.
 \end{equation}
 Hence the 2nd, $r+1$st and $r+2$nd edges are
$$\{a_{r+1},a_2,a_3 \ldots, a_{r}\}, \qquad \{a_{r+1}, a_{r+2}\ldots, a_{2r}\}, \qquad \{a_{2r+1},a_{r+2}, \ldots, a_{2r}\}.
$$
 The differences between the second and the first coordinates  in these three vectors are
$$d_1=a_{2}-a_{r+1} , \qquad  d_2= a_{r+2}-a_{r+1}, \qquad d_3= a_{r+2}-a_{2r+1}.$$
By~(\ref{j16}), it impossible for each of the three differences $d_1, d_2, d_3$ to be  powers of two.
This yields the lower bound in Theorem~\ref{cpthm} for $k\geq r+2$. \qed

\section{Proof of Theorem \ref{cgthm}}

In this section, we first apply Theorem \ref{splitting} to prove Theorem \ref{cgthm} for $k \leq 2r - 1$: we will show
\[ \ex_{\cir}(n,CP_k^r) \leq kr^{5r^2}n^{r-1}.\]
Since for $k,r \geq 2$, the extremal function $\ex(n,P_k^r)$ for an $r$-uniform tight path is $\Omega(n^{r - 1})$, and
$\ex_{\cir}(n,CP_k^r) \geq \ex(n,P_k^r)$, we have $\ex_{\cir}(n,CP_k^r) = \Theta(n^{r - 1})$ for $k \leq 2r - 1$. In the case $k = r + 1$, 
we have $$\ex_{\cir}(n,CP_k^r) \leq \ex_{\rightarrow}(n,CP_k^r) = {n \choose r} - {n - r \choose r}.$$ On the other hand,
$$\ex_{\cir}(n,CP_k^r) \geq \ex_{\cir}(n,CM_2^r) = {n \choose r} - {n - r \choose r},$$ so the second statement in Theorem \ref{cgthm} follows.
For $k \geq 2r$, we have $$\ex_{\cir}(n,CP_k^r) \leq \ex_{\rightarrow}(n,CP_k^r) = O(n^{r - 1}\log n)$$ from Theorem \ref{cpthm};
so  to prove Theorem \ref{cgthm} for $k \geq 2r$, we only need a matching construction.

\subsection{Upper bound for $k \leq 2r - 1$}

Given a convex geometric $r$-graph $H$ with $e(H) > kr^{5r^2}n^{r-1}$, we apply Theorem~\ref{splitting} to obtain a split subgraph $G \subset H$
where $e(G) > k v(G)^{r-1}$. Let $X_0 < X_1 < \dots < X_{r-3} < X$ be cyclic intervals such that every edge of $G$ contains two vertices in
$X$ and one vertex in each $X_i : 0 \leq i \leq r - 3$. Our main proposition is as follows:

\begin{prop} For $k \in [2r-1]$, $G$ contains a crossing $k$-path $v_0 v_1 \ldots v_{k+r-2}$ such that $v_i \in X_i$ for $i\not \equiv -1, -2 \mod r$ and $v_i \in X$ for $i \equiv -1, -2 \mod r$.
\end{prop}

\proof We proceed by induction on $k$, where the base case $k = 1$ is trivial. For the induction step, suppose that $1\le k \le 2r-2$, and we have proved the result for $k$ and we wish to prove it for $k+1$. Suppose that $k \equiv i \not\equiv 0, -1$ (mod $r$) where $i<r$. For each $f \in \partial G$ that has no vertex in $X_{i-1}$,  delete the edge $f \cup v \in G$
where $v$ is the largest vertex in $X_{i-1}$ in clockwise order. Let $G'$ be the subgraph that remains after deleting these edges. Then
$e(G')\ge e(G)-m^{r-1}>(k+1)m^{r-1}-m^{r-1}=km^{r-1}$, so by induction there is a $P_{k}^r$ in $G'$ with vertices $v_0, v_1, \ldots, v_{k-1},\ldots, v_{k+r-2}$, where $v_i \in X_i$ for $i\not \equiv -1, -2$ (mod $r$) and $v_i \in X$ for $i \equiv -1, -2$ (mod $r$). Let $v=v_{k+r-1}$ be the vertex in $X_{i-1}$ for which the edge $e_k=v_k v_{k+1} \ldots v_{k+r-1}$ was deleted in forming $G'$. Note that $v$ exists as $v_{k-1} v_k \ldots v_{k+r-2} \in E(G)$ and so $v_k \ldots v_{k+r-2} \in \partial G.$
Adding vertex $v$ and edge  $e_k$ to our copy of $P_k^r$ yields a copy of $P_{k+1}^r$ as required.

 Next suppose that $i \equiv 0,-1$ (mod $r$). In fact, we may assume that $k\in \{r-1, r\}$, and we are trying to add vertex $v=v_{k+r-1} \in \{v_{2r-2}, v_{2r-1}\}$ as above but now we want  $v \in X$. Suppose that $k=r-1$ and we are trying to add the vertex $v=v_{2r-2}$. Proceed exactly as before, except that when we have $f \in \partial G$ that has exactly one vertex in $X$, we choose $v$ to be the largest vertex in $X$ such that $v<v_{r-1}$ and $f \cup \{v\} \in G$. Such a $v$ certainly exists due to the edge  $e=f \cup \{v_{r-2}\} \in G$. For the case $k=r$, we choose $v$ to be the largest vertex in $X$ which again exists. \qed

\subsection{Lower bound for $k \geq 2r$}

We take the same family $G(n,r,r+2)$ as used for ordered hypergraphs (see Section \ref{orderedconstruction}),
but in the cyclic ordering of the vertex set.  When we have a $k$-edge crossing path $P=w_1w_2\ldots w_{r+k-1}$, the vertex $w_1$ does not need to be the leftmost in the first
edge $w_1\ldots w_r$, so the argument above does not go through for $k=r+2$. In fact,   $G(n,r,r+2)$ does contain
$CP_k^r$ for $k \leq 2r-1$. However, it does not have a crossing $(r+1)$-edge path in which the first vertex is the second left in
the first edge (repeating the argument in Subsection~\ref{orderedconstruction}). This implies that $G(n,r,r+2)$ does not
contain
$CP_{2r}^r$: If it has such a path $P=w_1\ldots w_{3r-1}$ and $w_1$ is the $i$th smallest in the first edge, then
$w_2$ is the $(i+1)$st smallest in the second edge and so on (modulo $r$). Thus, for some $1\leq j\leq r$, vertex $w_j$ is the second left
in the $j$th edge, and the subpath of $P$ starting from the $j$th edge has at least $r+1$ edges. \qed

\section{Concluding remarks}

$\bullet$ A hypergraph $F$ is a {\em forest} if there is an ordering of the edges $e_1,e_2,\dots,e_t$ of $F$ such that for all $i \in \{2,3,\dots,t\}$, there exists $h < i$ such that $e_i \cap \bigcup_{j < i} e_j \subseteq e_h$. It is not hard to show that $\ex(n,F) = O(n^{r - 1})$
 for each $r$-uniform forest $F$. It is therefore natural to extend the Pach-Tardos Conjecture~\ref{ptc} to $r$-graphs as follows:

\begin{conjecture}\label{main}
Let $r \geq 2$. Then for any ordered $r$-uniform forest $F$ with interval chromatic number $r$, $\ex_{\rightarrow}(n,F) = O(n^{r-1} \cdot \mbox{\rm polylog} \, n)$.
\end{conjecture}

Theorem \ref{splitting} shows that to prove Conjecture \ref{main}, it is enough to consider the setting of $r$-graphs of interval chromatic number $r$. Theorem \ref{cpthm}
verifies this conjecture for crossing paths, and also shows that the $\log n$ factor in Theorem \ref{splitting} is necessary. It would be interesting to find other general classes of ordered $r$-uniform forests for $r \geq 3$ for which Conjecture \ref{main} can be proved. A related problem is to determine for which ordered forests $F$ we have $\ex_{\to}(n, F)= O(n^{r-1})$? This is a hypergraph generalization of Bra\ss' question~\cite{Brass}. Theorem~\ref{splitting} can be used to prove this upper bound for many ordered forests other than just $CP_k^r$.

\medskip

$\bullet$ It appears to be substantially more difficult to determine the exact value of the extremal function for $r$-uniform crossing $k$-paths in the convex geometric setting than in the ordered setting. It is possible to show that for $k \leq 2r - 1$, 
\[ c(k,r) = \lim_{n \rightarrow \infty} \frac{\ex_{\cir}(n,CP_k^r)}{{n \choose r-1}}\]
exists. We have proved several nontrivial upper and lower bounds for $c(k,r)$ that will be presented in forthcoming work, however, we do not as yet know the value of $c(k,r)$ for any pair $(k,r)$ with $2 \leq k \leq r$, even though in the ordered
setting Theorem \ref{cpthm} captures the exact value of the extremal function for all $k\leq r+1$, and $c(r+1,r) = r$.

\medskip

$\bullet$  Let $CM_k$ denote the convex geometric graph consisting of $k$ pairwise crossing line segments. Capoyleas and Pach~\cite{Capoyleas-Pach} proved the following theorem which  extended a result of Ruzsa (he proved the case $k=3$) and settled a question of G\"artner and conjecture of Perles~\cite{Perles}:

\begin{thm} [Capoyleas-Pach~\cite{Capoyleas-Pach}] \label{cgg-cmatching}
	For all $n \geq 2k - 1$, $\ex_{\circlearrowright}(n,CM_k) = 2(k-1)n - {2k - 1 \choose 2}$.
\end{thm}

For $r \ge 2$, an $r$-uniform {\em crossing $k$-matching} $CM_k^r$ is an ordered $r$-graph whose vertex set is $v_0,v_1,\dots,v_{rk - 1}$, edges are  $\{v_i,v_{i+k},\dots,v_{i+(r-1)k}\}$ for $0 \leq i \leq k - 1$, and vertex ordering $v_0 < v_1 < \dots < v_{rk-1}$. The same definition works in the convex geometric
setting with $<$ a circular ordering of the vertices. Thus $CM_2^r$ is precisely the crossing matching $CM^r$.
It is not hard to see that $$\ex_{\circlearrowright}(n,CM_k^r) = \ex_{\rightarrow}(n,CM_k^r)$$ for all $n,k,r \geq 2$. We can prove that, unlike the results on the paths in Theorem \ref{cpthm}, there are no extra $\log n$ factors in the formulas for crossing matchings and we have $\ex_{\circlearrowright}(n,CM_k^r) = \Theta(n^{r-1})$. We will present sharper bounds in forthcoming work.

\paragraph{Acknowledgement.}

This research was partly conducted  during AIM SQuaRes (Structured Quartet Research Ensembles) workshops, and we gratefully acknowledge the support of AIM.

{\small

\begin{tabular}{ll}
\begin{tabular}{l}
{\sc Zolt\'an F\" uredi} \\
Alfr\' ed R\' enyi Institute of Mathematics \\
Hungarian Academy of Sciences \\
Re\'{a}ltanoda utca 13-15 \\
H-1053, Budapest, Hungary \\
E-mail:  \texttt{zfuredi@gmail.com}.
\end{tabular}
& \begin{tabular}{l}
{\sc Tao Jiang} \\
Department of Mathematics \\ Miami University \\ Oxford, OH 45056, USA. \\ E-mail: \texttt{jiangt@miamioh.edu}. \\
$\mbox{ }$
\end{tabular} \\ \\
\begin{tabular}{l}
{\sc Alexandr Kostochka} \\
University of Illinois at Urbana--Champaign \\
Urbana, IL 61801 \\
and Sobolev Institute of Mathematics \\
Novosibirsk 630090, Russia. \\
E-mail: \texttt {kostochk@math.uiuc.edu}.
\end{tabular} & \begin{tabular}{l}
{\sc Dhruv Mubayi} \\
Department of Mathematics, Statistics \\
and Computer Science \\
University of Illinois at Chicago \\
Chicago, IL 60607. \\
\texttt{E-mail: mubayi@uic.edu}.
\end{tabular} \\ \\
\begin{tabular}{l}
{\sc Jacques Verstra\"ete} \\
Department of Mathematics \\
University of California at San Diego \\
9500 Gilman Drive, La Jolla, California 92093-0112, USA. \\
E-mail: {\tt jverstra@math.ucsd.edu.}
\end{tabular}
\end{tabular}
}

\end{document}